\documentclass[12pt]{article}

\usepackage{amsmath, amsthm, amssymb, stmaryrd, color}

\newtheorem{theorem}{Theorem}[section]

\newtheorem{lemma}[theorem]{Lemma}
\newtheorem{corollary}[theorem]{Corollary}

\theoremstyle{definition}
\newtheorem{definition}[theorem]{Definition}

\begin{document}

\title{Undecidable problems for propositional \\ calculi with implication}

\author{Grigoriy V. Bokov\\
        \small Department of Mathematical Theory of Intelligent Systems\\
        \small Lomonosov Moscow State University\\
        \small Moscow, Russian Federation\\
        \small E-mail: bokov@intsys.msu.ru}

\maketitle

\begin{abstract}
In this article, we deal with propositional calculi over a signature containing the classical implication $\to$ with the rules of modus ponens and substitution. For these calculi we consider few recognizing problems such as recognizing derivations, extensions, completeness, and axiomatizations. The main result of this paper is to prove that the problem of recognizing extensions is undecidable for every propositional calculus, and the problems of recognizing axiomatizations and completeness are undecidable for propositional calculi containing the formula $x \to ( y \to x )$. As a corollary, the problem of derivability of a fixed formula $A$ is also undecidable for all $A$. Moreover, we give a historical survey of related results.
\end{abstract}

\maketitle

\section{Introduction}

In 1946, Tarski~\cite{Sinaceur:2000} proposed to consider decision problems for a propositional calculus, which is defined as a finite set of propositional formulas over some signature with a finite set of rules of inference. Many important and interesting problems arise for these calculi. For example, \emph{recognizing axiomatizations}, i.e., whether a given finite set of formulas constitutes (axiomatizes) an adequate axiom system for a propositional calculus, \emph{recognizing extensions}, i.e., whether a given finite set of formulas derives all theorems of propositional calculus, \emph{recognizing completeness}, i.e., whether a given finite set of theorems of propositional calculus constitutes an adequate axiom system for this calculus, and \emph{recognizing derivations}, i.e., whether a given formula derives from a propositional calculus. In this paper we consider only propositional calculi with the rules of modus ponens and substitution.

The first undecidable problem for propositional calculi was found by Linial and Post in 1949~\cite{LinialPost:49}. They proved the undecidability of recognizing completeness for the classical propositional calculus over the signature $\{ \neg, \vee \}$. Note that Linial and Post gave sketch of proof, the full proof of their result was restored later by Davis~\cite[pp.~137--142]{Davis:58} and Yntema~\cite{Yntema:64}. The Linial and Post theorem is an example of the first undecidable propositional calculus, i.e., the problem of recognizing derivations is undecidable for this calculus. As a corollary of this result, the problems of recognizing axiomatizations and extensions are also undecidable for the classical propositional calculus.

In 1963, Kuznetsov~\cite{Kuznetsov:63} proved the Linial and Post theorem for the intuitionistic calculus over the signature $\left\{ \neg, \vee, \&, \to \right\}$. Moreover, he obtained a much stronger result, that the problem of recognizing completeness, as well as the problems of recognizing axiomatisations and extensions, is undecidable not only for the intuitionistic, but also for every superintuitionistic propositional calculus, i.e., a finitely axiomatizable extension of the intuitionistic calculus. Particularly, the Kuznetsov theorem implies that the intuitionistic propositional calculus contains undecidable propositional calculi.

Several constructions of undecidable propositional calculi have been obtained. Singletary in 1964~\cite{Singletary:64:CPP} constructed an undecidable propositional calculus over the signature $\left\{ \neg, \to \right\}$. In 1965, Gladstone~\cite{Gladstone:65} and independently Ihrig~\cite{Ihrig:65:PLT} constructed propositional calculi for which the problem of recognizing derivations of formulas is of any required recursively enumerable degree of unsolvability. Note that Gladstone obtained the same result for every signature in which the implication is expressed. A much stronger result was obtained by Singletary in 1968~\cite{Singletary:68:RRA}. He constructed a pure implicational undecidable propositional calculus, i.e., calculus over the signature $\{\to\}$ whose axioms are derived from the axiom $x \to ( y \to x )$.

Kuznetsov noticed in~\cite{Kuznetsov:63} that A. A. Markov (Jr.) in 1961 proposed to consider the same class of recognizing problems for the implicational propositional calculus. In this way, Harrop in 1964~\cite{Harrop:64:RPPC} proved that the problem of recognizing completeness, as well as the problems of recognizing axiomatisations and extensions, is undecidable for every propositional calculus containing the formulas $x \to ( y \to x )$ and $x \to x$. Independently, in 1972 Bollman and Tapia~\cite{BollmanTapia:72:RUP} by using Singletary constructions~\cite{Singletary:68:RRA} proved the undecidability of the problem of recognizing extensions for the pure implicational fragment of the intuitionistic propositional calculus, i.e., the calculus with the following two axioms
\begin{equation*}
  \begin{array}{l}
    x \to (y \to x), \hbox{ and} \\
    (x \to (y \to z)) \to ((x \to y) \to (x \to z)).
  \end{array}
\end{equation*}
In 1994, Marcinkowski~\cite{Marcinkowski:94} obtained a much stronger result: fix an implicational propositional tautology $A$ that is not of the form $B \to B$ for some formula $B$, then the problem of recognizing extensions is undecidable for propositional calculus with the single axiom $A$. If we combine this with the Tarski result~\cite[p.~59]{Tarski:83}, we obtain that the problem of recognizing extensions is undecidable for every finitely axiomatizable extension of the propositional calculus with axioms $x \to (y \to x)$ and $x \to ( y \to ( ( x \to ( y \to z ) ) \to z ) )$.

Some recent observations of related results were given in 2014 by Zolin~\cite{Zolin:2013} and Bokov~\cite{Bokov:2015:URA}. Besides, an interesting observation was found by Chvalovsk\'{y}. He noted that the Linial and Post theorem for finitely represented superintuitionistic logics easily follows from Marcinkowski's construction in~\cite{Marcinkowski:94}.

The aim of this paper is to prove that the problem of recognizing extensions is undecidable for all propositional calculus and to show that a derivation of the formula $x \to ( y \to x )$ is sufficient for the undecidability of the problems of recognizing axiomatizations and completeness, i.e., every propositional calculus containing the formula $x \to ( y \to x )$ has undecidable problems of recognizing axiomatizations and completeness. As a corollary, the problem, whether a fixed formula $A$ is derivable from a given finite set of formulas by the rules of modus ponens and substitution, is also undecidable for all formula $A$, not only of the form $B \to B$ in contrast with the Marcinkowski result. Moreover, we consider a general methods of proving that a recognizing problem of propositional calculi is undecidable, and give a historical survey of related results.

This paper is organized as follows. In the next section we introduce the basic terminology and notation, give a historical survey of known results, and state our main result. Section 3 is devoted to a reduction of undecidable problems for propositional calculi. In the first part of this section we give a historical survey of methods to prove the undecidability of a recognizing problems for propositional calculi, describe a general method and illustrate it by examples. Next, we recall what a tag system is and formally reduce the halting problem of tag systems to the derivation problem of propositional calculi. In Section 4 we prove our results. Finally, in Section 5 we give some concluding remarks and discuss further researches

\section{Preliminaries and results}

We begin with some notation. Let us consider the language consisting of an infinite set of propositional variables $\mathcal{V}$ and the signature $\Sigma$, i.e., a finite set of connectives. Letters $x, y, z, u, p$, etc., are used to denote propositional variables. Usually connectives are binary or unary such as $\neg$, $\vee$, $\wedge$, or $\to$.

\emph{Propositional formulas} or $\Sigma$-\emph{formulas} are built up from the signature $\Sigma$ and propositional variables $\mathcal{V}$ in the usual way. For example, the following notations
\begin{equation*}
  x, \quad \neg A, \quad (A \vee B), \quad (A \wedge B), \quad (A \to B)
\end{equation*}
are formulas over the signature $\{\neg,\ \vee,\ \wedge,\ \to\}$. Capital letters $A, B, C$, etc., are used to denote propositional formulas. Throughout the paper, we will omit the outermost parentheses in formulas and parentheses assuming the customary priority of connectives.

Let $\Sigma$ be a signature containing the binary connective of implication $\to$. By a \emph{propositional calculus} or a $\Sigma$-\emph{calculus} we mean a finite set $P$ of $\Sigma$-formulas referred to as \emph{axioms} together with two rules of inference:

1) \emph{modus ponens}
\begin{equation*}
  A,~A \to B~\vdash~B;
\end{equation*}

2) \emph{substitution}
\begin{equation*}
  A~\vdash~\sigma A,
\end{equation*}
where $\sigma A$ is the substitution instance of $A$, i.e., the result of applying the substitution $\sigma$ to the formula $A$.

Denote by $[P]$ the set of derivable (or provable) formulas of a calculus $P$. A \emph{derivation} in $P$ is defined from the axioms and the rules of inference in the usual way. The statement that a formula $A$ is derivable from $P$ is denoted by $P \vdash A$.

Let us introduce the following pre-order relation on the set of all propositional calculus. We write $P_1 \leq P_2$ (or, equivalently, $P_2 \geq P_1$) if each derivable formula of $P_1$ is also derivable from $P_2$, i.e., if $[P_1] \subseteq [P_2]$. We write $P_1 \sim P_2$ and say that two calculi $P_1$ and $P_2$ are \emph{equivalent} if $[P_1] = [P_2]$. Finally, we write $P_1 < P_2$ if $[P_1] \subsetneq [P_2]$.

Now we formally define the problems of \emph{recognizing derivations} (\textbf{Drv}), \emph{extensions} (\textbf{Ext}), \emph{axiomatizations} (\textbf{Axm}), and \emph{completeness} (\textbf{Cmpl}) for a fixed $\Sigma$-calculus $P_0$:

\smallskip \noindent
\begin{center}
\begin{tabular}{ll}
  (\textbf{Drv}) & \emph{given a calculus $P$, determine whether $P_0 \geq P$;} \\
  (\textbf{Ext}) & \emph{given a calculus $P$, determine whether $P_0 \leq P$;} \\
  (\textbf{Axm}) & \emph{given a calculus $P$, determine whether $P_0 \sim P$;} \\
  (\textbf{Cmpl}) & \emph{given a calculus $P$ such that $P \leq P_0$, determine whether $P_0 \leq P$.}
\end{tabular}
\end{center}
\smallskip

\noindent
Denote by $\mathbf{Cl}_{\Sigma}$ the classical propositional calculus over a signature $\Sigma$, and by $\mathbf{Int}_{\Sigma}$ the intuitionistic propositional calculus over a signature $\Sigma$~\cite{Kleene:2002}.

The previous results can be summarized as follows.

\begin{theorem}[Linial and Post, 1949]
\textbf{Axm}, \textbf{Ext}, and \textbf{Cmpl} are undecidable for $\mathbf{Cl}_{\{\neg, \vee\}}$.
\end{theorem}

\begin{theorem}[Kuznetsov, 1963]
Fix a calculus $P_0 \geq \mathbf{Int}_{\{\neg, \vee, \wedge, \to\}}$, then \textbf{Axm}, \textbf{Ext}, and \textbf{Cmpl} are undecidable for $P_0$.
\end{theorem}

Consider the intuitionistic implicational propositional calculus $\mathbf{Int}_{\{\to\}}$ with the set of axioms~\cite[p.69]{HilbertBernays:68}:

\smallskip \noindent
\begin{center}
\begin{tabular}{ll}
  $(\mathrm{A}_1)$ & $x \to (y \to x)$,  \\
  $(\mathrm{A}_2)$ & $(x \to (y \to z)) \to ((x \to y) \to (x \to z))$. \\
\end{tabular}
\end{center}
\smallskip

\noindent The classical implicational propositional calculus $\mathbf{Cl}_{\{\to\}}$ is obtained from $\mathbf{Int}_{\{\to\}}$ by adding the Peirce law $((x \to y) \to x) \to x$~\cite[p.52]{Tarski:83}.

\begin{theorem}[Bollman and Tapia, 1964]
\textbf{Ext} is undecidable for $\mathbf{Int}_{\{\to\}}$.
\end{theorem}

\begin{theorem}[Marcinkowski, 1994] \label{T:Marcinkowski}
Fix a $\{\to\}$-tautology $A$ that is not of the form $B \to B$ for some formula $B$, then \textbf{Ext} is undecidable for the $\{\to\}$-calculus $\{A\}$.
\end{theorem}

Since the implicational calculi $\mathbf{Cl}_{\{\to\}}$ and $\mathbf{Int}_{\{\to\}}$ can be axiomatized by the following single formulas, as shown by {\L}ukasiewicz~\cite{Lukasiewicz:48} and Meredith~\cite{Meredith:53},
\begin{align*}
  \mathbf{Cl}_{\{\to\}}  & \sim \{ ((x \to y) \to z) \to ((z \to x) \to (u \to x)) \} \\
  \mathbf{Int}_{\{\to\}} & \sim \{ ((x \to y) \to z) \to (u \to ((y \to (z \to v)) \to (y \to v))) \}
\end{align*}
the following result also makes sense.

\begin{corollary}
\textbf{Axm}, \textbf{Cmpl} are undecidable for $\mathbf{Cl}_{\{\to\}}$, and \textbf{Ext} is undecidable for $\mathbf{Cl}_{\{\to\}}$ and $\mathbf{Int}_{\{\to\}}$.
\end{corollary}

In 1930, Tarski~\cite{Tarski:83} proved that every propositional calculus, which contains the formulas $x \to ( y \to x )$ and $x \to ( y \to ( ( x \to ( y \to z ) ) \to z ) )$, can be axiomatized by a single formula. Since these formulas are derivable from $\mathbf{Int}_{\{\to\}}$, we have the following corollary of the Marcinkowski result.

\begin{corollary} \label{C:Marchinkovski}
Fix a calculus $P_0 \geq \mathbf{Int}_{\{\to\}}$, then \textbf{Ext} is undecidable for $P_0$.
\end{corollary}

\begin{theorem}[Bokov and Marcinkowski~\footnote{According to the recent Chvalovsk\'{y} observation.}, 2014] \label{T:Bokov}
Fix a calculus $P_0 \geq \mathbf{Int}_{\{\to\}}$, then \textbf{Axm} and \textbf{Cmpl} are undecidable for $P_0$.
\end{theorem}

It is important to note that Corollary~\ref{C:Marchinkovski} and Theorem~\ref{T:Bokov} was obtained quite a long time ago by Harrop~\cite{Harrop:64:RPPC}.

\begin{theorem}[Harrop, 1964]
Fix a calculus $P_0 \geq \{ x \to (y \to x),\ x \to x\}$, then \textbf{Axm}, \textbf{Ext}, and \textbf{Cmpl} are undecidable for $P_0$.
\end{theorem}

Therefore, in order to prove the undecidability of the recognizing problem for a propositional calculus $P_0$, we must prove derivations of the formulas $x \to (y \to x)$ and $x \to x$ from $P_0$. In this paper we show that the second derivation, i.e., the derivation of the formula $x \to x$, is redundant to prove the undecidability of \textbf{Axm} and \textbf{Cmpl}. Indeed, as it was shown by Singletary in~\cite{Singletary:68:RRA} the derivation of the formula $x \to (y \to x)$ is sufficient to construct an undecidable propositional calculus.

\begin{theorem}[Singletary, 1968]
There exists a propositional calculus $P_0 \leq \{x \to (y \to x)\}$ for which \textbf{Drv} is undecidable.
\end{theorem}

\noindent Furthermore, we also prove that the derivation of the formula $x \to (y \to x)$ is redundant to prove the undecidability of \textbf{Ext}. Thus, our main result is the following theorem.

\begin{theorem} \label{T:main}
Fix a propositional calculus $P_0$, then

$(1)$\ \textbf{Ext} is undecidable for $P_0$;

$(2)$\ \textbf{Cmpl} is undecidable for $P_0$ if $P_0 \geq \{x \to (y \to x)\}$.
\end{theorem}

As corollary, we have the undecidability of the problem of recognizing axiomatizations.

\begin{corollary}
Fix a calculus $P_0 \geq \{x \to (y \to x)\}$, then \textbf{Axm} is undecidable for $P_0$.
\end{corollary}

\noindent Moreover, if we take in the theorem~\ref{T:main} the propositional calculus $P_0 = \{A\}$ for a $\Sigma$-formula $A$, then we obtain the undecidability of problem of derivability.

\begin{corollary}
Fix a signature $\Sigma \supseteq \{\to\}$ and a $\Sigma$-formula $A$, then the following problem is undecidable:
\begin{center}
given a $\Sigma$-calculus $P$, determine whether $P \vdash A$.
\end{center}
\end{corollary}

\noindent Particularly, this holds for a formula $A$ of the form $B \to B$ for some formula $B$ in contrast with Theorem~\ref{T:Marcinkowski}.

\section{Reduction of undecidable problems}

The typical method of proving a problem to be undecidable is a reduction of famous undecidable problem to this problem. In order to do this, it is sufficient to transform instances of an undecidable problem into instances of the new problem so that if a solution to the new problem were found, it could be used to decide the undecidable problem. Since we already know that no method can decide the old problem, no method can decide the new problem also.

\subsection{Historical survey}

One of the first problems to be proved undecidable is the halting problem of Turing machines~\cite{Turing:1937:CNA}. For example, Harrop~\cite{Harrop:64:RPPC} and Hughes~\cite{Hughes:76:TVIC} simulated Turing machines by implicational propositional calculi and reduced the halting problem to the decision problem of a partial implicational propositional calculus. Note that Hughes used only formulas contain at most two distinct variable symbols. But in some cases it is more convenient to reduce other undecidable problems.

Often decision problems for propositional calculi are associated with the word problem for semi-Thue systems. So, by a reduction of semi-Thue systems Yntema~\cite{Yntema:64} proved the undecidability of the completeness problem, Gladstone~\cite{Gladstone:65} and independently Ihrig~\cite{Ihrig:65:PLT} constructed calculi for which the problem of derivability of formulas is of any required recursively enumerable degree of unsolvability, Singletary~\cite{Singletary:68:RRA} constructed an undecidable implicational calculus, Boolman and Tapia~\cite{BollmanTapia:72:RUP} proved that it is impossible to algorithmically determine of an arbitrarily given partial propositional calculus whether or not the deduction theorem holds.

Numerous results were obtained on a simulation of Post normal system~\cite{Post:43} with the undecidable halting problem. For example, Linial and Post~\cite{LinialPost:49} noted that the undecidability of the completeness problem for the classical propositional calculus can be proved by a reduction of normal system introduced in~\cite{Post:44:RES}. In the same way Harrop~\cite{Harrop:58:EFM} proved existence of undecidable propositional calculus, Ratsa~\cite{Ratsa:89:UEP} proved the undecidability of the expressibility problem for modal logics. Recently, Zolin~\cite{Zolin:2013} obtained the Kuznetsov's results by a reduction of tag systems, i.e., a simple form of Post normal systems. A reduction of the halting problem of tag systems has been proposed by Bokov~\cite{Bokov:2009} for a proof of the Linial and Post theorem and improved in~\cite{Bokov:2015:URA} for a proof of the undecidability of some recognizing problems for propositional calculi with implication.

The above results are combined by using the halting condition of some computational machine such as Turing machine, semi-Thue system, or Post normal system. Another example of these machines is counter machines such as Minsky machines~\cite{Minsky:67:CFI}. Chagrov~\cite{Chagrov:97:ML} used Minsky machines to prove the undecidability of some problems of modal logics.

Nevertheless, there are reductions of other undecidable problems, not only the halting problem of some computational machine. So, Kuznetsov~\cite{Kuznetsov:63} devised a special calculus of primitive recursive functions, Marcinkowski~\cite{Marcinkowski:94} investigated the entailment problem for first-order Horn clauses.

\subsection{General method and examples}

In this section we describe a general method of reduction for undecidable problems of propositional calculi. First, for a given propositional calculus we must to fix
\begin{itemize}
  \item \emph{a model of computation} that is equivalent in its computational power to Turing machines, such as semi-Thue systems, Post normal systems, tag systems, or Minsky machines, and
  \item \emph{a procedure of encoding} that allows to encode operations of computation used in this model and their respective costs by a formulas of the propositional calculus.
\end{itemize}
Next, we must to simulate the model of computation by the inference process of the propositional calculus.

As an example, let us consider an abstract computational machine $T$, which deals with words over a finite alphabet $\mathcal{A}$. Operations of this machine are a finite set $R$ of pairs of words over $\mathcal{A}$.

A computation of the machine $T$ on an input word $\xi$ is a sequence of words $\lambda_0 = \xi, \lambda_1, \ldots$ such that every pair $(\lambda_i, \lambda_{i+1})$ is a instance of some operation from $R$ for all $i \geq 0$. Note that computations must be deterministic. We write $\xi \stackrel{T}{\Longmapsto} \zeta$ if there is a computation $\lambda_0, \lambda_1, \ldots, \lambda_n$, $n > 0$, such that $\lambda_0 = \xi$, $\lambda_n = \zeta$.

The halting condition of the machine $T$ is a finite set $H$ of words over $\mathcal{A}$. We say that the machine $T$ halts on input $\xi$ if the computation of $T$ on $\xi$ reaches a word from $H$, i.e., $\xi \stackrel{T}{\Longmapsto} \zeta$ for some $\zeta \in H$.

Next, let us consider propositional calculi with modus ponens and substitution. Assume that we want to prove an undecidability of the following recognizing problem: fix a class of propositional calculi $\mathcal{P}$ and a propositional calculus $P_0$, whether a given calculus $P \in \mathcal{P}$ contains $P_0$, i.e., $P \geq P_0$? In order to prove the undecidability of this problem, we fix a machine $T$ with the undecidable halting problem. Next, we encode words over $\mathcal{A}$ and construct a propositional calculus for the machine $T$ such that derivations of the codes of words simulates a computation of $T$ on them.

More precisely, let $\overline{\alpha}$ be the code of a word $\alpha \in \mathcal{A}^*$, and $\overline{\xi} \to \overline{\zeta}$ the code of a operation $(\xi, \zeta) \in R$. Usually, the code of any instance of operation $(\xi, \zeta)$ can be obtained from the code $\overline{\xi} \to \overline{\zeta}$ by substitution. We must construct a propositional calculus $P_T$ such that
\begin{enumerate}
  \item the computation of machine $T$ simulates as follows:
\begin{center}
  $\xi \stackrel{T}{\Longmapsto} \zeta$ \; iff\;  $P_T,~\overline{\xi}~\vdash~\overline{\zeta}$;
\end{center}
  \item the halting condition of $T$ defines as follows:
\begin{center}
  $T$ \emph{halts on} $\xi$ \; iff \; $P_T,~\overline{\xi}~\vdash~P_0$;
\end{center}
  \item a calculus obtained from $P_T$ by adding the axiom $\overline{\xi}$ is in $\mathcal{P}$.
\end{enumerate}
Then we obtain that the problem of recognizing, whether $P \geq P_0$ for a given calculus $P \in \mathcal{P}$, is undecidable, since otherwise the halting problem for $T$ is decidable.

In next sections we describe a process of reduction of the undecidable halting problem for abstract computational machines to a recognizing problem for propositional calculi in more details. For this reason, we take the tag system introduced by Post~\cite{Post:43} as an example of computational machine and consider propositional calculi over the signature $\Sigma$ such that $\{ \to \} \subseteq \Sigma$. For a given $\Sigma$-calculus $P_0$, tag system $T$ and a word $\xi$, we will effectively construct a $\Sigma$-calculus $P_{T,P_0,\xi}$ such that $T$ halts on the input word $\xi$ if and only if $P_0 \leq P_{T,P_0,\xi}$. Then the proof of Theorem~\ref{T:main} is immediately following from the undecidability of the halting problem~\cite{Minsky:61}.

First let us recall the notion of a tag system.

\subsection{Tag systems}

Let $\mathcal{A}$ be a finite alphabet of letters $a_1, \dots, a_m$. By $\mathcal{A}^*$ denote the set of all words over $\mathcal{A}$, including the empty word. For $\alpha \in \mathcal{A}^*$, denote by $|\alpha|$ the length of the word $\alpha$.

\begin{definition}[Post,~\cite{Post:43}]
A \emph{tag system} is a triple $T = \langle \mathcal{A}, \mathcal{W}, d \rangle$, where $\mathcal{A} = \{ a_1, \dots, a_m \}$ is a finite alphabet of $m$ symbols, $\mathcal{W} = \{ \omega_1, \dots, \omega_m\} \subseteq \mathcal{A}^*$ is a set of $m$ words, and $d \in \mathbb{N}$ is a \emph{deletion number}. Each words $\omega_i$ is associated to the letters $a_i$: $a_1 \to \omega_1, \dots, a_m \to \omega_m$.
\end{definition}

We say that $T$ is applicable to a word $\alpha \in \mathcal{A}^*$ if $|\alpha| \geq d$. The application of $T$ to a word $\alpha \in \mathcal{A}^*$ is defined as follows. Examine the first letter of the word $\alpha$. If it is $a_i$ then
\begin{enumerate}
  \item remove the first $d$ letters from $\alpha$, and
  \item append to its end the word $\omega_i$.
\end{enumerate}
Perform the same operation on the resulting word, and repeat the process as long as the resulting word has $d$ or more letters. To be precise, if $\alpha = a_i \beta \gamma$, $|\beta| = d-1$, and $\gamma \in \mathcal{A}^*$, then $T$ produces the word $\gamma \omega_i$ from the word $a_i \beta \gamma$. Denote this production by $a_i \beta \gamma \stackrel{T}{\longmapsto} \gamma \omega_i$. We write $\alpha \stackrel{T}{\Longmapsto} \beta$ if there are words $\gamma_1, \dots, \gamma_n$, $n \geq 1$, such that $\alpha = \gamma_1$, $\beta = \gamma_n$, and $\gamma_i \stackrel{T}{\longmapsto} \gamma_{i+1}$ for all $1 \leq i \leq n-1$.

Define the halting problem of tag systems. We say that a tag system $T$ \emph{halts} on a word $\alpha \in \mathcal{A}^*$ and write this as $T(\alpha)\downarrow$ if there exists a word $\beta \in \mathcal{A}^*$ such that $\alpha \stackrel{T}{\Longmapsto} \beta$ and $T$ is not applicable to $\beta$, i.e. $|\beta| < d$. The \emph{halting problem} for a fixed tag system $T$ is, given any word $\alpha \in \mathcal{A}^*$, to determine whether $T$ halts on $\alpha$.

\begin{theorem}[Minsky,~\cite{Minsky:61}] \label{T:Minsky}
There is a tag system $T$ for which the halting problem is undecidable.
\end{theorem}

Moreover, Wang~\cite{Wang:63} showed that this holds even for some tag system $T$ with $d = 2$ and $1 \leq |\omega_i| \leq 3$ for all $1 \leq i \leq m$. For this reason, throughout the paper we will assume that all words $\omega_i$ are nonempty.

\subsection{Encoding of letters and words}

Let $\mathcal{A}$ be a finite set $\{a_1, \dots, a_m\}$ as above. The set of all nonempty words over $\mathcal{A}$ is denoted by $\mathcal{A}^+$. We encode letters and words over $\mathcal{A}$ as one variable $\{\to\}$-formulas. In order to simulate a tag system over alphabet $\mathcal{A}$ correctly, this encoding must be an injective function between words over $\mathcal{A}$ and their codes. As shown in~\cite{Bokov:2009}, a word-to-formula encoding is related with difficulties of derivation of a code of one word from a code of other word. Below we show that it is more convenient to encode a word as a set of formulas. Moreover, we give a one-to-one (bijective) encoding between words over $\mathcal{A}$ and their codes.

Fix a one-variable $\{\to\}$-formula $\hat{x}$. As an example of $\hat{x}$ may be the formula $x \to x$ or $x \to (x \to x)$. Note that $\hat{x}$ is an arbitrary $\{\to\}$-formula with a single variable $x$. For future use we introduce a shortcut for the following formula of two variables:
\begin{equation*}
  x \circ y := ( ( \hat{y} \to \hat{y} ) \to \hat{y} ) \to ( \hat{x} \to ( ( \hat{y} \to \hat{y} ) \to \hat{y} ) ).~\footnote{Note that $\hat{y}$ is the substitution instance of $\hat{x}$ with replacing $x$ by $y$.}
\end{equation*}
It is obvious that $x \circ y$ is a substitution instance of the axiom $x \to (y \to x)$. The following lemma is needed for the sequel.
\begin{lemma} \label{L:NonUnifiableFormula}
$x \circ y$ and $(x \circ y) \to z$ are not unifiable.
\end{lemma}
\noindent The proof is straightforward and left to the reader.

Now we define the notion of code of a letter. First, let us fix a unique variable $p$. Then the \emph{code} of a letter $a_i \in \mathcal{A}$, for $1 \leq i \leq m$, is a formula:
\begin{equation} \label{F:CodeOfLetter}
  \overline{a_i} := \big(\underbrace{p \to (p \to \dots (p}_i \to p ))\big) \circ p.
\end{equation}
Since $x \circ y$ is a substitution instance of the axiom $x \to (y \to x)$, we have the following lemma.

\begin{lemma} \label{L:DerivabilityOfCodeOfLetter}
$x \to (y \to x) \vdash \overline{a}$, for every letter $a \in \mathcal{A}$.
\end{lemma}

In order to encode a word, i.e. finite sequence of letter, we must define an operation of concatenation for letters. For this reason, we introduce a shortcut $x \cdot y$ as an abbreviation for the formula $((x \to x) \to x) \circ y$. Thus we come to the following definition.

\begin{definition} (Zolin,~\cite{Zolin:2013}) \label{D:AlphabeticFormula}
An \emph{alphabetic formula} over the alphabet $\mathcal{A}$, or an $\mathcal{A}$-\emph{formula} for short, is an arbitrary $\{\cdot\}$-formula over the codes of letters from $\mathcal{A}$. Formally, $\overline{a}$ is a $\mathcal{A}$-formula for each letter $a \in \mathcal{A}$, and if $A$, $B$ are $\mathcal{A}$-formulas then so is $A \cdot B$.
\end{definition}

An example of an alphabetic formula is $(\overline{a} \cdot (\overline{c} \cdot \overline{e})) \cdot (\overline{e} \cdot \overline{a})$. It is easily seen that every $\mathcal{A}$-formula is associated with a word over $\mathcal{A}$. To every $\mathcal{A}$-formula $A$ we associate its $word(A) \in \mathcal{A}^+$ by induction: $word(\overline{a}) := a$ for each letter $a \in \mathcal{A}$, and $word(A \cdot B)$ := $word(A)word(B)$. For example, the $\mathcal{A}$-formulas $(\overline{a} \cdot \overline{e}) \cdot (\overline{c} \cdot \overline{a})$ and $(\overline{a} \cdot (\overline{e} \cdot \overline{c})) \cdot \overline{a}$ are associated with the same word $aeca$.

Let us introduce some notation that will be useful later. Given a formula $A$, denote by $A^*$ the set of all substitution instances of $A$. Similarly, given a set $M$ of formulas, denote by $M^*$ the set
\begin{equation*}
  M^* := \bigcup_{A \in M} A^*.
\end{equation*}
We call two formulas $A$ and $B$ \emph{unifiable} if $A^* \cap B^* \neq \emptyset$. For example, formulas $x \to (y \to z)$ and $(y \to z) \to x$ are unifiable, but formulas $x \to (y \to x)$ and $(y \to x) \to x$ are not unifiable.

\begin{lemma} \label{L:AlphabeticFormulas}
No two distinct $\mathcal{A}$-formulas are unifiable.
\end{lemma}
\begin{proof}
By induction on the definition of an $\mathcal{A}$-formula $A$.

\noindent Let $A$ be the code of a letter $a_i \in \mathcal{A}$. There are two cases:

1) If $B$ is the code of a letter $a_j \in \mathcal{A}$, then $i \neq j$. Denote by $C_i$ the following formula
\begin{equation*}
  \underbrace{p \to (p \to \dots (p}_i \to p )).
\end{equation*}
Then $A$ is the formula $C_i \circ p$ and $B$ is the formula $C_j \circ p$. Since $C_i$ and $C_j$ are not unifiable for $i \neq j$, we conclude that $A$ and $B$ are not unifiable.

2) Let $B$ is a formula $B_1 \cdot B_2$ for some $\mathcal{A}$-formulas $B_1$ and $B_2$. Then $A$ is the formula $C_i \circ p$ and $B$ is the formula $((B_1 \to B_1) \to B_1) \circ B_2$. Since the formulas $C_i$ and $(x \to x) \to x$ are not unifiable for all $i$, $1 \leq i \leq m$, we have that $A$ and $B$ are not unifiable, .

Now let $A = A_1 \cdot A_2$ for some $\mathcal{A}$-formulas $A_1$ and $A_2$, so it can be assumed that $B = B_1 \cdot B_2$ for some $\mathcal{A}$-formulas $B_1$ and $B_2$. If $A$, $B$ are unifiable, then also $A_1$, $B_1$ and $A_2$, $B_2$ are unifiable. By induction hypothesis, $A_1 = B_1$ and $A_2 = B_2$. Hence, $A = B$.

This completes the proof of the lemma.
\end{proof}

Finally, we define the \emph{code} of a word $\alpha \in \mathcal{A}^+$ as the finite set $\overline{\alpha}$ consisting of all $\mathcal{A}$-formulas associated with the word $\alpha$. Formally,
\begin{equation*}
  \overline{\alpha} := \{ A \mid A \text{ is a } \mathcal{A}\text{-formula such that } word(A) = \alpha \}.
\end{equation*}

\noindent Note that the code of a letter $a \in \mathcal{A}$ is the formula defined as in~(\ref{F:CodeOfLetter}), but the code of a single-letter word $a \in \mathcal{A}^+$ is the set consisting of the code of letter $a$. Throughout this paper, we will use the same notation for the code of a letter and the code of a single-letter word.

As an example, $\{\overline{a} \cdot \overline{c}\}$ is the code of the word $ac$, $\{\overline{a} \cdot (\overline{c} \cdot \overline{e}),\ (\overline{a} \cdot \overline{c}) \cdot \overline{e}\}$ is the code of the word $ace$, and $\{\overline{a} \cdot (\overline{c} \cdot (\overline{e} \cdot \overline{c})),\ \overline{a} \cdot ((\overline{c} \cdot \overline{e}) \cdot \overline{c}),\ ((\overline{a} \cdot \overline{c}) \cdot (\overline{e} \cdot \overline{c})),\ (\overline{a} \cdot (\overline{c} \cdot \overline{e})) \cdot \overline{c},\ ((\overline{a} \cdot \overline{c}) \cdot \overline{e}) \cdot \overline{c}\}$ is the code of the word $acec$, where $a, c, e \in \mathcal{A}$.

Since every $\mathcal{A}$-formula is a substitution instance of the axiom $x \to (y \to x)$, we have the following generalization of Lemma~\ref{L:DerivabilityOfCodeOfLetter}.

\begin{lemma} \label{L:DerivabilityOfAFormulas}
$x \to (y \to x) \vdash \overline{\alpha}$, for every word $\alpha \in \mathcal{A}^+$.
\end{lemma}

Similarly, we call two codes $\overline{\alpha}$ and $\overline{\gamma}$ \emph{unifiable} if $\overline{\alpha}^* \cap \overline{\gamma}^* \neq \emptyset$. Lemma~\ref{L:AlphabeticFormulas} implies:
\begin{corollary}
No two distinct codes are unifiable.
\end{corollary}

Now we introduce the following convention. In order to simplify a notation of formulas, we will use an abbreviation $\overline{\alpha}$ for some word $\alpha \in \mathcal{A}^+$ as a part of formulas. For example, a formula $\overline{\alpha} \to x$ is a shortcut for the following set of formulas
\begin{equation*}
  \{ A \to x \mid A \in \overline{\alpha} \},
\end{equation*}
and a formula $\overline{\alpha} \cdot x \to x \cdot \overline{\beta}$ is a shortcut for the set
\begin{equation*}
  \{ (A \cdot x) \to (x \cdot B) \mid A \in \overline{\alpha},\ B \in \overline{\beta} \}.
\end{equation*}
Note that all alphabetic formulas are one-variable formulas with the same variable $p$, so we substitute the same formula in different occurrences of alphabetic formulas. As an example of this substitution let us consider a formula $A$ of the form
\begin{equation*}
\overline{\alpha}[p] \cdot x \to x \cdot \overline{\beta}[p],
\end{equation*}
where square brackets denote a dependence on variables or subformulas. Then a substitution instance of $A$ is any formula of the form
\begin{equation*}
  \overline{\alpha}[B] \cdot C \to C \cdot \overline{\beta}[B]
\end{equation*}
for some formulas $B$ and $C$.

\subsection{Simulation of tag systems}

For a given tag system $T$ we construct a propositional $\{\to\}$-calculus $P_T$ such that the derivation of codes of words in $P_T$ simulates productions of words in $T$.

Let $T = \langle \mathcal{A}, \mathcal{W}, d \rangle$, where $\mathcal{A} = \{ a_1, \dots, a_m \}$, $\mathcal{W} = \{ \omega_1, \dots, \omega_m \}$, and $d \in \mathbb{N}$. Recall that all $\omega_i$ are assumed to be nonempty. Denote by $P_T$ a $\{\to\}$-calculus with the following groups of axioms.

\medskip
\noindent \emph{Productions} of the tag system $T$:
$$
  \leqno
  \begin{aligned}
  & (\mathrm{T}_1) \quad \overline{a_i \alpha} \cdot x \to x \cdot \overline{\omega_i} \\
  & (\mathrm{T}_2) \quad \overline{a_i \alpha} \to \overline{\omega_i}
  \end{aligned}
  \quad \text{ for all } \alpha \in \mathcal{A}^+,\ |\alpha| = d-1,\ 1 \leq i \leq m;
$$

\noindent \emph{Transformation rules}:
$$
  \leqno
  \begin{aligned}
    & (\mathrm{R_1}) \quad x \cdot (y \cdot z) \to (x \cdot y) \cdot z \\
    & (\mathrm{R_2}) \quad (x \cdot y) \cdot z \to x \cdot (y \cdot z) \\
    & (\mathrm{R_3}) \quad (x \cdot (y \cdot z)) \cdot u \to ((x \cdot y) \cdot z) \cdot u \\
    & (\mathrm{R_4}) \quad ((x \cdot y) \cdot z) \cdot u \to (x \cdot (y \cdot z)) \cdot u
  \end{aligned}
$$

Define two subsystems of the calculus $P_T$:
\begin{equation*}
  \mathrm{T} := \mathrm{T}_1 \cup \mathrm{T}_2, \quad \mathrm{R} := \mathrm{R}_1 \cup \mathrm{R}_2 \cup \mathrm{R}_3 \cup \mathrm{R}_4.
\end{equation*}
Since they are rather weak and not even capable to derive $A \to C$ from $A \to B$
and $B \to C$, we introduce the following useful notation: $P \vdash A \Rightarrow B$ if and only if there are formulas $C_0 = A,\ C_2,\ \ldots,\ C_{n-1},\ C_n = B$, $n \geq 0$, such that $P \vdash C_i \to C_{i+1}$ for all $0 \leq i \leq n-1$.

Since every formula $A \cdot B$ is a substitution instance of the axiom $x \to (y \to x)$, we have the following lemma.

\begin{lemma} \label{L:Inclusion}
$P_T \leq \{x \to (y \to x)\}$.
\end{lemma}

Now we prove some properties of the calculus $P_T$.

\subsubsection{Derivability of the $T$-productions}

Here we show that the calculus $P_T$ can ``simulate'' productions of the tag system $T$. At the beginning let us prove auxiliary lemmas.

\begin{lemma} \label{L:CodeDerivation}
$\mathrm{R} \vdash A \Rightarrow \overline{\alpha}$, for all $\alpha \in \mathcal{A}^+$ and $A \in \overline{\alpha}$.
\end{lemma}
\begin{proof}
Let $\alpha = a_1 \ldots a_n$. Since all axioms in $\mathrm{R}$ are invertible, i.e., $B \to A \in \mathrm{R}$ whenever $A \to B \in \mathrm{R}$, it is sufficient to prove that
\begin{equation*}
  \mathrm{R} \vdash A \Rightarrow \overrightarrow{\alpha},
\end{equation*}
where $\overrightarrow{\alpha}$ is the following formula $\overline{a_1} \cdot (\overline{a_2} \cdot \ldots \cdot (\overline{a_{n-1}} \cdot \overline{a_n}))$.

Without loss of generality it can be assumed that $n \geq 3$. We split up the proof into two steps. First, we will show that there exists an alphabetic formula $B$ such that
\begin{equation*}
  \mathrm{R} \vdash A \Rightarrow B \cdot \overrightarrow{\xi},
\end{equation*}
where $\xi \in \mathcal{A}^+$ and $\alpha = \beta \xi$, for $\beta = word(B)$. Next, we will prove that
\begin{equation*}
  \mathrm{R} \vdash B \cdot \overrightarrow{\xi} \Rightarrow \overrightarrow{\beta \xi}
\end{equation*}
by induction on $|\beta|$.

\textbf{Proof of the first step:} Since $n \geq 3$, there is an integer $k \geq 1$, nonempty words $\alpha_1, \ldots, \alpha_k, \xi$ and alphabetic formulas $A_1, \ldots, A_k$ such that $\alpha = \alpha_1 \ldots \alpha_k \xi$,
\begin{equation*}
  A = A_1 \cdot (A_2 \cdot \ldots \cdot ( A_k \cdot \overrightarrow{\xi})),
\end{equation*}
and $word(A_i) = \alpha_i$, for $1 \leq i \leq k$. Denote by $B$ the following formula $(((A_1 \cdot A_2) \cdot A_3) \cdot \ldots \cdot A_k)$, then we have
\begin{equation*}
  \mathrm{R} \vdash A \Rightarrow B \cdot \overrightarrow{\xi}
\end{equation*}
by a multiple application of axiom $\mathrm{(R_1)}$.

\textbf{Proof of the second step:} By induction on $|\beta|$, where $\beta = word(B)$. If $|\beta| = 1$, then the formulas $B \cdot \overrightarrow{\xi}$ and $\overrightarrow{\beta \xi}$ are identical.

Now let $|\beta| \geq 2$, so there is an integer $m \geq 1$, a letter $a \in \mathcal{A}$, and nonempty words $\beta_1, \ldots, \beta_m$ such that $\beta = \beta_1 \ldots \beta_m a$ and
\begin{equation*}
  B = B_1 \cdot (B_2 \cdot \ldots \cdot (B_m \cdot \overline{a})),
\end{equation*}
for some alphabetic formulas $B_1, \ldots, B_m$ such that $word(B_i) = \beta_i$, $1 \leq i \leq m$. Denote by $C$ the following formula $(((B_1 \cdot B_2) \cdot B_3) \ldots \cdot B_m)$. Then we derive in $R$:
\begin{equation*}
  B \cdot \overrightarrow{\xi} \quad \stackrel{\mathrm{(R_3)}}{\Longrightarrow} \quad
  (C \cdot \overline{a}) \cdot \overrightarrow{\xi} \quad \stackrel{\mathrm{(R_2)}}{\longrightarrow} \quad
  C \cdot \overrightarrow{a \xi} \quad \stackrel{\mathrm{IH}}{\Longrightarrow} \quad
  \overrightarrow{\beta \xi}
\end{equation*}
where the first derivation is a multiple application of axiom $\mathrm{(R_3)}$, the second derivation is a single application of axiom $\mathrm{(R_2)}$, and the last derivation uses induction hypothesis for the word $\gamma$ such that $\gamma =  word(C)$. Note that $\gamma$ is exactly the word $\beta_1 \ldots \beta_m$. The lemma is proved.
\end{proof}

\begin{lemma} \label{L:Derivability}
If $\xi \stackrel{T}{\longmapsto} \zeta$ then $P_T \vdash \overline{\xi} \Rightarrow \overline{\zeta}$, for all $\xi, \zeta \in \mathcal{A}^+$.
\end{lemma}
\begin{proof}
Since $T$ is applicable to $\xi$, we have $|\xi| \geq d$. Therefore, $\xi = a_i \alpha \beta$ and $\zeta = \beta \omega_i$, where $|\alpha| = d-1$ and $|\beta| \geq 0$.

If $|\beta| = 0$, then $P_T \vdash \overline{\xi} \to \overline{\zeta}$ by the axiom $(\mathrm{T}_2)$.

Let $|\beta| > 0$, so we derive in $P_T$:
\begin{equation*}
  \overline{\xi} \quad \stackrel{\mathrm{L}}{\Longrightarrow} \quad
  \overline{a_i \alpha} \cdot \overline{\beta} \quad \stackrel{\mathrm{(T_1)}}{\longrightarrow} \quad
  \overline{\beta} \cdot \overline{\omega_i} \quad \stackrel{\mathrm{L}}{\Longrightarrow} \quad
  \overline{\zeta}
\end{equation*}
where the first and last derivations are due to Lemma~\ref{L:CodeDerivation}, and the second derivation is the substitution instance of the axiom $\mathrm{(T_1)}$. The lemma is proved.
\end{proof}

\begin{corollary} \label{C:Derivability}
If $\xi \stackrel{T}{\Longmapsto} \zeta$ then $P_T \vdash \overline{\xi} \Rightarrow \overline{\zeta}$, for all $\xi, \zeta \in \mathcal{A}^+$.
\end{corollary}

\noindent The proof is trivial by definition of the tag system.

\subsubsection{Production of the $P_T$-derivations}

Here we show that the tag system $T$ can produce, on the input word, the words whose codes have derivations in $P_T$. As a preliminary let us introduce some notation and prove auxiliary lemmas.

Given $\alpha \in \mathcal{A}^+$, denote by $T_{\alpha}$ the set of all $\mathcal{A}$-formulas whose words have productions of the tag system $T$ on the input word $\alpha$:
\begin{equation*}
  T_{\alpha} = \{ A \mid A \text{ is a } \mathcal{A}\text{-formula such that } \alpha \stackrel{T}{\Longmapsto} word(A) \}.
\end{equation*}
It is clear that $\overline{\alpha} \subseteq T_{\alpha}$ for all $\alpha \in \mathcal{A}^+$.

\begin{lemma} \label{L:Intersection}
$P_T^* \cap T_{\alpha}^* = \emptyset$, for all $\alpha \in \mathcal{A}^+$.
\end{lemma}
\noindent The proof is trivial by application of Lemma~\ref{L:NonUnifiableFormula}.

For any propositional calculus $P$, denote by $\left\langle P \right\rangle$ the set of propositional formulas obtained from $P$ by applying modus ponens and substitution once:
\begin{align*}
  \left\langle P \right\rangle := & \left\{ B \mid A, A \to B \in P \text{ for some formula } A \right\} \cup \\
  & \left\{ \sigma A \mid A \in P \text{ and } \sigma \text{ is a substitution} \right\}.
\end{align*}
Furthermore, let $\left\langle P \right\rangle_0 = P$ and
\begin{equation*}
  \left\langle P \right\rangle_{n+1} = \left\langle \left\langle P \right\rangle_n \right\rangle
\end{equation*}
for $n \geq 0$. It follows easily that $\left\langle P \right\rangle_n \subseteq \left\langle P \right\rangle_{n+1}$ for all $n \geq 0$ and the set $[P]$ of all derivable formulas of the calculus $P$ can be represented as
\begin{equation*}
  [P] = \left\langle P \right\rangle_{\infty} = \bigcup_{n \geq 0} \left\langle P \right\rangle_n.
\end{equation*}
Let $A$ be a formula derivable from $P$. We say that $A$ has the \emph{derivation height} $n$, if $A \in \left\langle P \right\rangle_n$ and $A \notin \left\langle P \right\rangle_{n-1}$.

The following theorem describes formulas derivable from the calculus $P_T$ and the code of a nonempty word $\alpha \in \mathcal{A}^+$.

\begin{lemma} \label{L:Production}
$[ P_T \cup \overline{\alpha} ] =  P_T^* \cup T^*_{\alpha}$  for all $\alpha \in \mathcal{A}^+$.
\end{lemma}
\begin{proof}
It is evident that
\begin{equation*}
  P_T^* \cup T^*_{\alpha} \subseteq [ P_T \cup \overline{\alpha} ]
\end{equation*}
by Lemma~\ref{L:CodeDerivation} and Corollary~\ref{C:Derivability}, so we only prove by induction on the derivation height $n \geq 0$ that
\begin{equation*}
  \left\langle P_T \cup \overline{\alpha} \right\rangle_n \subseteq P_T^* \cup T^*_{\alpha}.
\end{equation*}

If $n = 0$, then $\left\langle P_T \cup \overline{\alpha} \right\rangle_0 = P_T \cup \overline{\alpha}$. Clearly, $\overline{\alpha} \subseteq T^*_{\alpha}$ and all axioms of $P_T$ are in $P_T^*$.

Let the induction assumption be satisfied for some $n \geq 1$. Since the right-hand side of the inclusion is closed under substitution, we only consider the case of a formula $B$ obtained by modus ponens from some formulas $A,\ A \to B \in \left\langle P_T \cup \overline{\alpha} \right\rangle_n$. By induction hypothesis,
\begin{equation*}
  \left\langle P_T \cup \overline{\alpha} \right\rangle_n \subseteq P_T^* \cup T^*_{\alpha}.
\end{equation*}
It is easily shown that $P_T^* \cap T^*_{\alpha} = \emptyset$ due to Lemma~\ref{L:Intersection}. Hence either $A$ or $A \to B$ are in $T^*_{\alpha}$, since otherwise $A$ is both a substitution instance of $x \circ y$ and $x \circ y \to z$, which is impossible by Lemma~\ref{L:NonUnifiableFormula}. If $A \to B \in T^*_{\alpha}$, then $A$ is a substitution instance of the formula $(y \to y) \to y$. However, $A \in P_T^* \cap T^*_{\alpha}$, which is impossible, because all formulas in $P_T^*$ and $T^*_{\alpha}$ are not unifiable with $(y \to y) \to y$. Therefore, $A \in T^*_{\alpha}$ and $A \to B \in P_T^*$.

Now we show that $B \in T^*_{\alpha}$. Since $A \in T^*_{\alpha}$, then $A \in \overline{\gamma}^*$ for some word $\gamma \in \mathcal{A}^+$ such that $\alpha \stackrel{T}{\Longmapsto} \gamma$. Note that $A \to B$ is a substitution instance of some axiom in $P_T$, so we need to consider the following two cases.

\textbf{Case 1.} $A \to B$ is a substitution instance of an axiom in $\mathrm{T}$. Then $A$ is a substitution instance of the formula $\overline{a_i \alpha_1} \cdot C$ or $\overline{a_i \alpha_1}$ for some letter $a_i \in \mathcal{A}$, a word $\alpha_1 \in \mathcal{A}^*$ with $|\alpha_1| = d-1$, and a formula $C$. Since $A \in \overline{\gamma}^*$, we have that $C$ is a alphabetic formula. Therefore, by Lemma~\ref{L:AlphabeticFormulas} there is a unique word $\alpha_2 \in \mathcal{A}^*$ such that $\gamma = a_1 \alpha_1 \alpha_2$. It is clear that $\alpha_2 = word(C)$ and $B$ is the substitution instance of the alphabetic formula $C \cdot \overline{\omega_i}$ or $\overline{\omega_i}$. Thus, $B \in \overline{\eta}^*$ for $\eta = \alpha_2 \omega_i$ and $\gamma \stackrel{T}{\longmapsto} \eta$.

\textbf{Case 2.} $A \to B$ is a substitution instance of an axiom in $\mathrm{R}$. Since the formula $A$ is a substitution instance of an alphabetic formula $C \in \overline{\gamma}$ and the set of alphabetic formulas $\overline{\gamma}$ is closed under application modus ponens and the axioms $\mathrm{R}$, we have that also $B \in \overline{\gamma}^*$.

These cases exhaust all possibilities and so we have that $B \in \overline{\eta}^*$ for some word $\eta \in \mathcal{A}^*$ such that $\gamma \stackrel{T}{\Longmapsto} \eta$. Hence $B \in T^*_{\alpha}$, since $\alpha \stackrel{T}{\Longmapsto} \gamma$ by induction hypothesis. The proof is completed.
\end{proof}

Now we prove that the code of each nonempty word over $\mathcal{A}$ derivable from $P_T$ and $\overline{\alpha}$ is the code of a word produced from $\alpha$ by the tag system $T$.

\begin{corollary} \label{C:Production}
If $P_T \vdash \overline{\xi} \Rightarrow \overline{\zeta}$ then $\xi \stackrel{T}{\Longmapsto} \zeta$, for all $\xi, \zeta \in \mathcal{A}^+$.
\end{corollary}
\begin{proof}
Let $P_T \vdash \overline{\xi} \Rightarrow \overline{\zeta}$, so $\overline{\zeta} \in [P_T \cup \overline{\xi}]$. Then $\overline{\zeta} \in P_T^* \cup T^*_{\xi}$ by Lemma~\ref{L:Production}. Since $\overline{\zeta} \notin P_T^*$ due to Lemma~\ref{L:NonUnifiableFormula}, we have that $\overline{\zeta} \in T^*_{\xi}$ and so $\xi \stackrel{T}{\Longmapsto} \zeta$ by definition of the set $T_{\xi}$. The lemma is proved.
\end{proof}

\subsubsection{Halting condition}

Above we shown that derivations in the propositional calculus $P_T$ can simulate productions in the tag system $T$. Now we describe how to perform the halting condition of tag system $T$ on input words. For this reason, we consider a propositional calculus $P_0$ and the following group of axioms.

\medskip
\noindent \emph{The halting condition} for the tag system $T$:
$$
  \leqno (\mathrm{H}) \quad \overline{\alpha} \to A \quad \text{ for all } \alpha \in \mathcal{A}^+,\ |\alpha| < d,\ A \in P_0.
$$

Denote by $P_{T, P_0}$ the calculus $P_T \cup \mathrm{H}$, and by $P_{T, P_0, \xi}$ the calculus $P_T \cup \mathrm{H} \cup \overline{\xi}$. Let the tag system $T$ halts on the input word $\xi$, we take the minimal $n \geq 0$ such that $\left\langle P_{T, P_0, \xi} \right\rangle_n$ contains at least one substitution instance of element of the code for some word $\zeta \in \mathcal{A}^+$ with $|\zeta| < d$:
\begin{equation*}
  N_{\xi} := \min \left\{ n \geq 0 \mid \overline{\zeta}^* \cap \left\langle P_{T, P_0, \xi} \right\rangle_{n} \neq \emptyset, \text{ for some } \zeta \in \mathcal{A}^+ \text{ with } |\zeta| < d \right\}.
\end{equation*}
If $T$ does not halt, then we put $N_{\xi} := \infty$. We have the following generalization of Lemma~\ref{L:Production}.

\begin{lemma} \label{L:FormOfDerivations}
$\left\langle P_{T,P_0,\xi} \right\rangle_{N_{\xi}} \subseteq P_{T,P_0}^* \cup T^*_{\xi}$.
\end{lemma}
\begin{proof}
Clearly, it is sufficient to consider the case of the proof of Lemma~\ref{L:Production} for which $A \to B$ is a substitution instance of axioms $(\mathrm{H})$. But this case is impossible, since otherwise we would have that $0 < |\gamma| < d$. This contradicts to the fact that
\begin{equation*}
  \overline{\gamma}^* \cap \left\langle P_{T,P_0,\xi} \right\rangle_{n} \neq \emptyset
\end{equation*}
and $n < N_{\xi}$. The lemma is proved.
\end{proof}

Now we prove the key lemma of this section.

\begin{lemma} \label{L:HaltingCondition}
Fix a propositional calculus $P_0$, then the tag system $T$ halts on input $\xi$ if and only if $P_{T,P_0,\xi} \geq P_0$, for all $\xi \in \mathcal{A}^+$.
\end{lemma}
\begin{proof}
By definition, if the tag system $T$ halts on an input word $\xi \in \mathcal{A}^+$, then $\xi \stackrel{T}{\Longmapsto} \zeta$ for some word $\zeta \in \mathcal{A}^+$ such that $|\zeta| < d$. Since
\begin{equation*}
  P_T \vdash \overline{\xi} \Rightarrow \overline{\zeta}
\end{equation*}
by Corollary~\ref{C:Derivability}, the code $\overline{\zeta}$ of $\zeta$ is derivable from $P_T$ and $\overline{\xi}$. If we recall that $P_{T,P_0}$ contains the axioms $\overline{\zeta} \to A$ for each $A \in P_0$, we obtain that $P_{T,P_0,\xi} \vdash A$ and so $P_{T,P_0,\xi} \geq P_0$.

Conversely, let $P_{T,P_0,\xi} \geq P_0$. Recall that the formula $x \circ y$ is built up with using a fixed formula $\hat{x}$ as follows:
\begin{equation*}
  x \circ y = ( ( \hat{y} \to \hat{y} ) \to \hat{y} ) \to ( \hat{x} \to ( ( \hat{y} \to \hat{y} ) \to \hat{y} ) ).
\end{equation*}
Since $\hat{x}$ is an arbitrary one-variable $\{\to\}$-formula, we may assume that every formula in $P_0$ is not a substitution instance of $x \circ y$ or $x \circ y \to z$. On the other hand, all formulas having derivations in $P_{T,P_0,\xi}$ of a height less or equal $N_{\xi}$ is a substitution instances of $x \circ y$ or $x \circ y \to z$ by Lemma~\ref{L:FormOfDerivations}. Hence, if $T$ does not halt on input $\xi$, then $N_{\xi} = \infty$ and, therefore, $P_{T,P_0,\xi} \ngeq P_0$. This contradiction completes the proof.
\end{proof}

\section{The proof of Theorem~\ref{T:main}}

\subsection{Undecidability of recognizing extensions}

If \textbf{Ext} is decidable for a $\Sigma$-calculus $P_0$, then the following problem is decidable: given a tag system $T$ and a word $\xi \in \mathcal{A}$, determine whether $P_0 \leq P_{T,P_0,\xi}$. By Lemma~\ref{L:HaltingCondition}, the decidability of the last problem for the calculus $P_0$ is equivalent to the decidability of the halting problem for the tag system $T$. Since the halting problem of tag systems is undecidable by Theorem~\ref{T:Minsky}, this contradiction completes the proof of undecidability of recognizing extensions.

\subsection{Undecidability of recognizing completeness}

If $\{x \to (y \to x)\} \leq P_0$, then $P_{T,P_0,\xi} \leq P_0$ by Lemmas~\ref{L:DerivabilityOfAFormulas} and~\ref{L:Inclusion}. Hence the problem of recognizing completeness of $P_0$ reduces to the problem of recognizing extensions of $P_0$, which is undecidable. This completes the proof of the theorem.

\section{Conclusion and further research}

In this paper, we established the undecidability of the problem of recognizing extensions for all propositional calculus, and the undecidability of the problem of recognizing completeness, as well as axiomatizations, for all propositional calculus whose theorems contain the formula $x \to (y \to x)$. These results were obtained for the signatures containing the symbol of implication $\to$. It is easily shown that the proofs remain valid, with minor changes, if we consider a signature, which does not contain the symbol~$\to$, but there is some propositional formula having $x, y$ as sole variables, whose truth-table interpretation is ``$x$ implies $y$''.

The other observation is that we can redefine encoding of letters and words by using the formula $x \to (F \to x)$ instead of the formula $x \to (y \to x)$, where $F$ is an arbitrary formula not containing the variable $x$. If we replace the key formula $x \circ y$ with the following formula
\begin{equation*}
  ( ( \hat{y} \to \hat{y} ) \to \hat{y} ) \to ( \hat{F}[x] \to ( ( \hat{y} \to \hat{y} ) \to \hat{y} ) ),
\end{equation*}
where $\hat{F}[x]$ is the substitution instance of $F$ by replacing all occurrences of variables with a fixed one-variable formula $\hat{x}$, we obtain the following interesting generalization of Theorem~\ref{T:main}.

\begin{theorem} \label{T:Conclusion}
Fix a propositional formula $F$ not containing the variable $x$ and a propositional calculus $P_0 \geq \{x \to (F \to x)\}$, then \textbf{Axm} and \textbf{Cmpl} are undecidable for $P_0$.
\end{theorem}

\noindent We leave the proof to the reader.

A natural and interesting question arises with respect to this generalization: there is an enumerable set of propositional formulas $M$ for which the condition $[P_0] \cap M \neq \emptyset$ holds if and only if \textbf{Axm} and \textbf{Cmpl} are undecidable for $P_0$. Since Gladstone in~\cite{Gladstone:79:DOPC} proved that \textbf{Drv} is decidable for every one-variable propositional calculus, it seems to be interesting to consider only formulas containing two or more variables. Theorem~\ref{T:Conclusion} shows that two-variables formulas are sufficient.

\section{Acknowledgement}

The author is grateful to Karel Chvalovsk\'{y} for discussion of undecidable problems of propositional calculi investigated by Marcinkowski and useful comments that improved the manuscript.


\end{document}